\documentclass{amsart}
\usepackage{amssymb}
\usepackage{amsmath}
\usepackage{comment}
\usepackage{mathtools}
\usepackage[colorlinks]{hyperref}
\hypersetup{citecolor=blue}

 \usepackage{amsthm}
 \usepackage{amsfonts}
 \usepackage{verbatim}

\newtheorem{theorem}{Theorem}[section]
\newtheorem{corollary}[theorem]{Corollary}
\newtheorem{lemma}[theorem]{Lemma}
\newtheorem{prop}[theorem]{Proposition}

\newtheorem{proposition}[theorem]{Proposition}

\newtheorem{conjecture}[theorem]{Conjecture}

\DeclareMathOperator{\End}{End}

\DeclareMathOperator{\Tr}{Tr}
\DeclareMathOperator{\Aut}{Aut}
\DeclareMathOperator{\Cl}{Cl}
\DeclareMathOperator{\GL}{GL}
\DeclareMathOperator{\SL}{SL}

\DeclareMathOperator{\Gal}{Gal}
\DeclareMathOperator{\norm}{Norm}
\DeclareMathOperator{\ord}{\upsilon}

\newcommand{\Frob}{\mathrm{Frob}}

\newcommand{\cA}{\mathcal{A}}

\newcommand{\cG}{\mathcal{G}}
\newcommand{\cH}{\mathcal{H}}

\newcommand{\OO}{{\mathcal O}}
\newcommand{\ga}{\mathfrak{a}}

\newcommand{\ff}{\mathfrak{f}}

\newcommand{\fp}{\mathfrak{m}}
\newcommand{\gp}{\mathfrak{p}}
\newcommand{\fq}{\mathfrak{q}}
\newcommand{\fr}{\mathfrak{r}}
\newcommand{\mP}{\mathfrak{P}}

\newcommand{\ZK}{\mathbb{Z}_K}

\newcommand{\C}{\mathbb C}

\newcommand{\R}{\mathbb R}
\newcommand{\Q}{\mathbb Q}
\newcommand{\Z}{\mathbb Z}
\newcommand{\F}{\mathbb F}
\newcommand{\A}{\mathbb A}
\newcommand{\fN}{\mathfrak{N}}
\def\O{{\mathcal O}}

\def\p{{\mathfrak p}}
\def\q{{\mathfrak q}}

\begin{document}

\title[asymptotic Fermat]{On the asymptotic Fermat's Last Theorem\\ over number fields}
\author{Mehmet Haluk \c{S}eng\"un}
\address{\normalfont{University of Sheffield,
School of Mathematics and Statistics, Hicks Building,
Sheffield, UK}}
\email{m.sengun@sheffield.ac.uk}
\author{Samir Siksek}
\address{\normalfont{University of Warwick,
Mathematics Institute, Coventry, UK}}
\email{samir.siksek@gmail.com}

 \begin{abstract} 
Let $K$ be a number field, $S$ be the set of primes of $K$
above $2$ and $T$ the subset of primes above $2$ having inertial
degree $1$. Suppose that $T \ne \emptyset$, and moreover, that
for every
solution $(\lambda,\mu)$ to the $S$-unit equation
\[
\lambda+\mu=1, \qquad \lambda,~\mu \in \OO_S^\times,
\]
there is some $\mP \in T$ such that 
$\max\{ \ord_\mP(\lambda),\ord_\mP(\mu)\} \le 4 \ord_\mP(2)$.
Assuming two deep but standard conjectures 
from the Langlands programme, we prove the asymptotic
Fermat's Last Theorem over $K$: there is some $B_K$
such that for all prime exponents $p>B_K$ the only solutions
to $x^p+y^p+z^p=0$ with $x$, $y$, $z \in K$ satisfy $xyz=0$.
We deduce that the asymptotic Fermat's Last Theorem holds
for imaginary quadratic fields $\Q(\sqrt{-d})$
with $-d \equiv 2$, $3 \pmod{4}$ squarefree.
\end{abstract}
 
\maketitle

\section{Introduction}
 Dickson, in his \emph{History of the Theory of Numbers}
(\cite[pages 758 and 768]{Dickson}) gives a survey
of early work on the Fermat equation over number fields,
with the earliest reference being 
to  the work of Maillet (1897).
Over a period of almost a century, 
number theorists have intermittently sought extensions of Kummer's 
cyclotomic approach 
to the
setting of number fields.
Perhaps the most satisfying work
in that direction is that of Hao and Parry \cite{HP},
who prove several results on the Fermat equation
over quadratic fields subject to a regularity condition
on the prime exponent $p$ (as for $\Q$ one does not know how to 
prove that there are infinitely many regular primes).

In view of Wiles' remarkable proof of Fermat's Last Theorem,
it is now more natural to attack the Fermat equation 
over number fields via Frey curves and modularity.
Jarvis and Meekin \cite{JMee}
did just this, proving
Fermat's Last Theorem over $\Q(\sqrt{2})$. They were followed
by Freitas and Siksek \cite{FS2} who proved Fermat's Last Theorem
for various
real quadratic fields of small discriminant. In another work,
Freitas and Siksek \cite{FS} proved the asymptotic version 
of Fermat's Last Theorem (explained below) for totally real fields satisfying
some auxiliary conditions. Key to these successes is the 
extraordinary progress in modularity over totally real fields,
due to the efforts of Barnett-Lamb, Breuil, Diamond, Gee, Geraghty, Kisin,
Skinner, Taylor, Wiles, and others. Alas our understanding of
modularity (or automorphy) in the setting of general
number fields is largely conjectural. One can ask if it is
possible to replicate the aforementioned successes for the 
Fermat equation over general number fields, by assuming 
standard conjectures. The purpose of this paper
is to address this question, and to highlight 
additional
challenges that arise in the general number field setting.

\medskip

Let $K$ be an algebraic number field.
To keep this Introduction self-contained we relegate the precise
statements of the two conjectures we assume to later sections,
and now only briefly indicate what they are.
\begin{itemize}
\item Conjecture~\ref{conj:Serre}: this is a weak version of
 Serre's modularity conjecture (\cite{gee_etal}) for odd, irreducible, 
continuous $2$-dimensional mod $p$ representations of Gal$(\overline{\Q}/K)$ 
that are finite flat at every prime over $p$. 
\item Conjecture~\ref{conj:ES}: this is a conjecture
in the Langlands Programme (see \cite{taylor})
which says that every weight $2$ newform (for $\GL_2$) over $K$ with 
integer Hecke eigenvalues has an associated elliptic 
curve over $K$ or a fake elliptic curve over $K$. 
\end{itemize}

To state our main result, we need to set up some notation.
Write $\ZK$ for the ring of integers of $K$.
Let $S$ for the set of primes $\mP$ of $\ZK$ above $2$,
and let $T$ be the subset of $\mP \in S$ 
with inertial degree $1$ (or equivalently with residue class field $\F_2$). 
We consider the Fermat
equation
\begin{equation}\label{eqn:Fermat}
x^p+y^p+z^p=0
\end{equation}
with $x$, $y$, $z \in K$ and prime exponent $p$. 
We say that a solution $(x,y,z)=(a,b,c) \in K^3$
is non-trivial if $abc \ne 0$. 
 
\begin{theorem}\label{thm:FermatGen}
Let $K$ be a number field for which Conjectures~\ref{conj:Serre}
and~\ref{conj:ES} hold. 
Let $S$, $T$ be as above and suppose $T \ne \emptyset$.
Write $\OO_S^\times$ for the set of $S$-units of $K$.
Suppose that for every solution $(\lambda,\mu)$ to the $S$-unit equation
\begin{equation}\label{eqn:sunit}
\lambda+\mu=1, \qquad \lambda,\, \mu \in \OO_S^\times \, .
\end{equation}
there is
some $\mP \in T$ that satisfies
$\max\{ \lvert \ord_{\mP} (\lambda) \rvert, \lvert \ord_{\mP}(\mu) \rvert \} 
\le 4 \ord_{\mP}(2)$.
Then the asymptotic Fermat's Last Theorem holds for $K$:
 there is some constant $B_K$ such that 
the Fermat equation \eqref{eqn:Fermat}
has no non-trivial solutions with prime exponent $p>B_K$.
\end{theorem}

\subsection{Differences from the totally real case}
The reader comparing the statement of our Theorem~\ref{thm:FermatGen}
with that of Theorem 3 of Freitas and Siksek \cite{FS} may incorrectly
(but understandably) presume that the proof is largely the same.
In fact, in addition to making use of ideas in \cite{FS} 
we need to deal with following two additional challenges that do
not arise in the totally real case.
\begin{enumerate}
\item[(i)] For a general number field $K$, Serre's
modularity conjecture relates a representation
$G_K \rightarrow \GL_2(\F_p)$, subject to certain conditions,
to a mod $p$ eigenform of weight $2$ over $K$. If $K$ is totally real,
such a mod $p$ eigenform lifts to a complex eigenform over $K$; this is not
generally
the case for a number field $K$ with complex embeddings.
We show that this difficulty is circumvented
in our asymptotic Fermat setting where the prime exponent $p$
is assumed to be sufficiently large. This step makes the
constant $B_K$ in Theorem~\ref{thm:FermatGen} ineffective, in contrast
to the totally real case. To make this effective we would
need effective bounds for the size of torsion subgroups
of integral cohomology groups associated to
certain locally symmetric spaces (see Section~\ref{sec:lift}).
\item[(ii)] If $K$ has a real embedding, then a weight $2$ 
complex eigenform over $K$ with rational eigenvalues conjecturally corresponds
to an elliptic curve over $K$. 
This is not true if $K$ is totally complex; the
eigenform does sometimes correspond to a fake elliptic curve.
A careful study of images of inertia at primes $\mP \in T$ of the 
mod $p$ representation of the Frey
curve shows that they are incompatible with 
images of inertia for fake elliptic curves.
\end{enumerate}

\subsection{An Octic Example}
We stress that $S$-unit equations have finitely
many solutions and that  there is a practical algorithm
for determining these solutions;
see for example \cite{Smart}. Thus the criterion in Theorem~\ref{thm:FermatGen}
is algorithmically testable. To illustrate this, take
$K=\Q(\zeta_{16})$ where $\zeta_{16}$ is a primitive
$16$-th root of unity. Then $K$ is a totally complex
number field of degree $8$.
Let $\mP=(1-\zeta_{16}) \cdot \Z_K$.
Then $2 \Z_K= \mP^8$. It follows that $S=T=\{\mP\}$.
Smart 
\cite[Section 5]{Smart2} 
determines the solutions to the $S$-unit equation \eqref{eqn:sunit}
for this particular field and finds that there are precisely
$795$ solutions $(\lambda,\mu)$. It turns out that the largest possible
value of 
$\max\{\lvert \ord_\mP(\lambda)\rvert,\lvert\ord_\mP(\mu)\rvert\}$ is $22$, which is smaller
than $4 \ord_\mP(2)=32$. By Theorem~\ref{thm:FermatGen},
assuming Conjectures~\ref{conj:Serre} and~\ref{conj:ES},
the asymptotic Fermat's Last Theorem holds for $K$. 

\subsection{Imaginary Quadratic Fields}
Let $K=\Q(\sqrt{-d})$ be an imaginary quadratic field,
where $d$ is a squarefree positive integer. If $-d \equiv 5 \pmod{8}$
then $2$ is inert in $K$ and so $T =\emptyset$ and 
Theorem~\ref{thm:FermatGen} does not apply.
If $-d \equiv 1 \pmod{8}$ then $2$ splits in $K$ and if 
$-d \equiv 2$ or $3 \pmod{4}$ then it ramifies. Here
we consider the particularly simple case of $-d \equiv 2$
or $3 \pmod{4}$.

\begin{theorem}
Let $K=\Q(\sqrt{-d})$ be an imaginary quadratic field with
where $d$ is a squarefree positive integer 
satisfying $-d \equiv 2$ or $3 \pmod{4}$.
Assume Conjectures~\ref{conj:Serre}
and~\ref{conj:ES}.  Then the asymptotic Fermat's Last Theorem holds for $K$.
\end{theorem}
\begin{proof}
Note that $S=T=\{\mP\}$ where $\mP^2=2 \ZK$. 
Suppose first that $d>2$.
The assumptions ensure that the only units in $K$ are $\pm 1$.
If $\mP=(a+b\sqrt{-d})$ is principal then
$a^2+d b^2=2$ giving a contradiction. 
Thus $\mP$ is not principal.
Hence if $(\lambda,\mu)$ is any solution to the $S$-unit equation \eqref{eqn:sunit}
then $\lambda=\pm 2^r$, $\mu=\pm 2^s$ with $r$, $s \in \Z$. We quickly deduce that
$(\lambda,\mu)=(2,-1)$ or $(-1,2)$ or $(1/2,1/2)$. In particular, all solutions satisfy
$\max\{ \lvert \ord_{\mP} (\lambda) \rvert, \lvert \ord_{\mP}(\mu) \rvert \} 
\le 4 \ord_{\mP}(2)$.
The proof is complete by Theorem~\ref{thm:FermatGen} for $d>2$.
The cases $d=1$, $2$ are similar.
\end{proof}
It is straightforward, though somewhat lengthy,
 to adapt the method of \cite[Sections 6--7]{FS}
to deduce that the asymptotic Fermat's Last Theorem holds
for $5/6$ of imaginary quadratic fields, assuming 
Conjectures~\ref{conj:Serre} and~\ref{conj:ES}.

We are indebted to the referee for suggesting several
corrections.

\section{Eigenforms for $\GL_2$ over number fields} \label{sec:modularforms}
In this section, we discuss modular forms, both complex and mod $p$, from a 
perspective that will be most useful for us. 
Let $K$ be an algebraic number field with ring of integers $\ZK$ and signature
$(r,s)$. Let $\hat{\Z}_K$ be the finite ad\`eles of $\ZK$ and let $\A_K,
\A^f_K$ denote the rings of ad\`eles and of finite ad\`eles of $K$, respectively. We
let $\mathcal{H}_2^{\pm}$ denote the union of the upper and lower half planes
and $\mathcal{H}_3$ denote the hyperbolic 
$3$-space. Then $\GL_2(K)$ acts on 
$X=\left ( \mathcal{H}_2^{\pm} \right )^r \times \mathcal{H}_3^s$ 
via the embedding $\GL_2(K) \hookrightarrow \GL_2(K \otimes \R) \simeq \GL_2(\R)^r \times \GL_2(\C)^s$.
Fix an ideal $\mathfrak{N} \subseteq \ZK$ and define the compact open subgroup 
\[
U_0(\fN):= \left \{ \gamma \in \GL_2(\widehat{\Z}_K):\gamma\equiv\begin{pmatrix}*&*\\ 0&*\end{pmatrix} \bmod \mathfrak{N} \right \}. 
\]
Consider the adelic locally symmetric space
\[
Y_0(\fN) = \GL_2(K) \backslash \left ( \left ( \GL_2(\A^f_K) / U_0(\fN) \right ) \times X \right ).
\] 
This space is a disjoint union of Riemannian $(2r+3s)$-folds
\[
Y_0(\fN) = \bigsqcup_{j=1}^h \Gamma_j \backslash X
\]
where $\Gamma_j$ are arithmetic subgroups of $\GL_2(K)$, with $\Gamma_1$ being the usual congruence subgroup 
$\Gamma_0(\fN)$ of the modular group $\GL_2(\ZK)$, and $h$ is the class number of $K$.

 For $i \in \{ 0, \hdots, 2r+3s \}$, consider the $i$-th cohomology group
$H^i(Y_0(\fN), \C)$.
For every prime $\q$ coprime to the level $\fN$, we can construct a linear
endomorphism $T_\q$ of $H^i(Y_0(\fN), \C)$ (called a Hecke operator) and these
operators commute with each other. We let $\mathbb{T}^{(i)}_\C(\fN)$ denote the
commutative $\Z$-algebra generated by these Hecke operators inside the
endomorphism algebra of $H^i(Y_0(\fN), \C)$. 

For the purposes of this paper, a {\bf (weight $2$) complex eigenform $f$ over
$K$ of degree $i$ and level $\fN$} is a ring homomorphism $\ff:
\mathbb{T}^{(i)}_\C(\fN) \rightarrow \C$.  Note that the values of $\ff$ are
algebraic integers and they generate a number field which we shall denote
$\Q_\ff$. We shall call a complex eigenform {\bf trivial} if we have
$\ff(T_\q)= \mathbf{N}\q +1$ for all primes $\q$ coprime to the
level\footnote{In the setting of $\GL_2$, non-triviality amounts to
cuspidality.}. We call two complex eigenforms $\ff,\mathfrak{g}$ with possibly
different degrees and levels \textbf{equivalent} if $\ff(T_\q) =
\mathfrak{g}(T_\q)$ for almost all prime ideals $\q$ (notice that the two
Hecke operators $T_\q$ may live in different Hecke algebras). A complex eigenform, say of
level $\fN$, is called {\bf new} if it is not equivalent to one whose level is
a proper divisor of $\fN$. 

Now let $p$ be a rational prime unramified in $K$ and coprime to the level. The
cohomology group $H^i(Y_0(\fN),\overline{\F}_p)$ also comes equipped with Hecke
operators, still denoted $T_\q$ (we only consider these for primes $\q$ 
coprime to $p\fN$). We shall denote the corresponding algebra by
$\mathbb{T}^{(i)}_{\overline{\F}_p}(\fN)$. A {\bf (weight $2$) mod $p$
eigenform $\theta$ over $K$ of degree $i$ and level $\fN$} is a ring
homomorphism $\theta : \mathbb{T}^{(i)}_{\overline{\F}_p}(\fN) \rightarrow
\overline{\F}_p$.

%

\subsection{Lifting mod $p$ eigenforms}\label{sec:lift}
We say that a mod $p$ eigenform $\theta$, say of level $\fN$, lifts to a
complex eigenform if there is a complex eigenform $\ff$ of the same degree and
level and a prime ideal $\mathfrak{p}$ of $\Q_\ff$ over $p$ such
that for every prime $\q$ of $K$ coprime to $p\fN$ we have 
$\theta(T_\q) \equiv \overline{\ff(T_\q)} \pmod{\mathfrak{p}}$. 

A very intriguing aspect of the theory is that in general mod $p$ eigenforms do
not lift to complex ones. The obstruction to lifting is given by $p$-torsion in
the integral cohomology as we now explain.  The long exact sequence associated
to the multiplication-by-$p$ short exact sequence $0 \rightarrow \Z
\xrightarrow{\cdot p} \Z \rightarrow \F_p \rightarrow 0$ gives rise to the
following short exact sequences
$$0 \rightarrow H^i(Y_0(\fN),\Z) \otimes \F_p \rightarrow H^i(Y_0(\fN),\F_p) \rightarrow H^{i+1}(Y_0(\fN), \Z)[p] \rightarrow 0,$$
where $H^{i+1}(Y_0(\fN), \Z)[p]$ denotes the $p$-torsion subgroup of $H^{i+1}(Y_0(\fN), \Z)$. Hence we see that $p$-torsion of $H^{i+1}(Y_0(\fN), \Z)$ vanishes if and only if the reduction map 
from $H^i(Y_0(\fN),\Z)$ to $H^i(Y_0(\fN),\F_p)$ is surjective.  Now, the
existence of an eigenform (complex or mod $p$) is equivalent to the existence
of a class in the corresponding cohomology group that is a simultaneous
eigenvector for the Hecke operators such that its eigenvalues match the values
of the eigenform. With this interpretation, we can utilize the lifting lemmas
of Ash and Stevens \cite[Section 1.2]{ash-stevens} and deduce that every mod
$p$ eigenform of degree $i$ lifts to a complex one when $H^{i+1}(Y_0(\fN), \Z)$
has no $p$-torsion.

The integral cohomology groups $H^i(Y_0(\fN), \Z)$ are well known to be finitely generated. Thus for a given level $\fN$, there are only finitely many primes $p$ for which there is an $i$ such that 
$H^i(Y_0(\fN),\Z)[p]$ is non-trivial. We obtain the following easy corollary which is crucial for our paper. 
\begin{proposition} \label{prop: lifting} There is a constant
$B$, depending only on $\fN$, such that for any prime $p>B$, every mod $p$ eigenform of level $\fN$ lifts to a complex one.
\end{proposition}

\section{Mod $p$ Galois representations} \label{sec:Reps}
We will be using the following very special case of Serre's modularity
conjecture over number fields. This conjecture concerns the modularity of
$2$-dimensional mod $p$ Galois representations.  While it is easy to predict
the level and the Nebentypus of the sought after mod $p$ eigenform (Serre's
original recipe \cite{SerreDuke} is still applicable), 
predicting all the possible weights (which actually is
a completely local issue) is a very difficult task. A general weight recipe for
$\GL_2$ over number fields was given\footnote{Originally given for totally real fields but as the problem of 
weights is a local issue, their recipe applies to any number field.} 
by Buzzard, Diamond and Jarvis \cite{BDJ} (see also \cite[Section 6]{bergeron-venkatesh} and \cite{gee_etal}). 
However we shall not need the full strength of their conjecture; the mod $p$ Galois
representations that we shall encounter in this paper are of a very special
type, namely finite flat at every prime over $p$, and for such representations it is well-known
that (again going back to Serre's original work) we should expect the trivial
Serre weight (which we called \lq weight $2$\rq\ in this paper) 
among the possible weights. This is sufficient for our purposes.

Recall that for every real embedding $\sigma: K \hookrightarrow \R$ and every extension $\tau: \overline{K} \rightarrow \C$ of $\sigma$, 
we obtain a {\em complex conjugation} \  $\tau^{-1} \circ c \circ \tau \in G_K$, where $\langle c \rangle = {\rm Gal}(\C / \R)$. 
We say that $\overline{\rho} : G_K \rightarrow \GL_2(\overline{\F}_p)$ is {\em odd} if the determinant of every complex conjugation is $-1$. 
If $K$ is totally complex, we will regard $\overline{\rho}$ automatically as odd.

\begin{conjecture} \label{conj:Serre}
Let $\overline{\rho} : G_K \rightarrow \GL_2(\overline{\F}_p)$ be an odd, 
irreducible, continuous representation with Serre conductor $\fN$ (prime-to-$p$
part of its Artin conductor) and trivial character (prime-to-$p$ part of
${\det}(\overline{\rho})$). 
Assume that $p$ is unramified in $K$ and that $\overline{\rho}\vert_{G_{K_\p}}$ arises from a finite-flat group scheme over
$\Z_{K_\p}$ for every prime $\p | p$. Then there is a (weight $2$) mod $p$ eigenform $\theta$ over $K$ of level
$\fN$ such that for all primes $\q$ coprime to $p\fN$, we have 
\[
\Tr(\overline{\rho}
({\Frob}_\q)) = \theta(T_\q). 
\]
\end{conjecture} 


\section{Motives attached to complex eigenforms} \label{sec: abelian varieties}
Recall that a simple abelian surface $A$ over $K$ whose algebra 
$\End_K(A)\otimes_\Z \Q$ of $K$-endomorphisms is an indefinite division quaternion algebra $D$ over $\Q$ is commonly called a {\em fake elliptic curve}.  The field of definition of a fake elliptic curve is necessarily totally complex. 

Let $A/K$ be a fake elliptic curve and let $p$ be a prime of good reduction
for $A$. Consider the representation $\sigma_{A,p} :  G_K \longrightarrow
\GL_4 (\Z_p)$ coming from the $p$-adic Tate module $T_p (A)$ of $A$.
Let $\O$  denote $\End_K(A)$ viewed as an order in $D$. Assume that $p$
splits $D$ and denote $\O \otimes \Z_p$ by $\O_p$. Then $\O$ acts on
$T_p (A)$ via endomorphisms and moreover $T_p(A) \simeq \O_p$ as a
left $\O$ module. Consequently $\Aut_{\O_p}(T_p(A)) \cong \O_p^\times$,
giving us a $2$-dimensional representation.
\begin{equation}\label{eqn:fec-rep}
\rho_{A,p} : \ G_K \longrightarrow \O_p^\times \simeq \GL_2(\Z_p).
\end{equation}
By a theorem of Ohta \cite{Ohta}, we have
$\sigma_{A,p} \simeq  \rho_{A,p} \oplus \rho_{A,p}$.
It is the $2$-dimensional representation $\rho_{A,p}$ 
that will be of interest to us.

The following is a very special case of a fundamental conjecture of the
Langlands Programme \cite{clozel, taylor} which asserts the existence of
motives associated to cohomological automorphic representations.

\begin{conjecture}\label{conj:ES}
  Let $\ff$ be a (weight $2$) complex eigenform over $K$ of level $\mathfrak{N}$ that is non-trivial and new. If $K$ has some real place, then there exists an elliptic curve $E_\ff/K$, of conductor 
  $\mathfrak{N}$, such that
  \begin{align}\label{eq: property defining E_f}
    \# E_\ff(\ZK/\fq) = 1 + {\bf N}\q - \ff(T_\q) 
\quad \text{for all $\q \nmid \fN$}.
  \end{align}
If $K$ is totally complex, then there exists either an elliptic curve $E_\ff$ of conductor $\fN$ satisfying \eqref{eq: property defining E_f} or a fake elliptic curve $A_\ff/K$, of conductor 
$\fN^2$, such that
  \begin{align}\label{eq: property defining A_f}
    \# A_\ff(\ZK/\q) = \left (1 + {\bf N}\q - \ff(T_\q) \right )^2 
\quad \text{for all $\q \nmid \fN$}.
  \end{align}
\end{conjecture}

Finally, we record a standard fact about fake elliptic curves that we shall crucially use later, see \cite[Section 3]{jordan}.
\begin{theorem}\label{thm:goodredn}
 Let $A/K$ be a fake elliptic curve. Then $A$ has potential good
reduction everywhere. More precisely, let $\q$ be a prime of $K$ and consider
$A/ K_\q$. There is totally ramified extension $K'/ K_\q$ of degree dividing
$24$ such that $A/K'$ has good reduction.
\end{theorem}

\section{The Frey curve and the associated mod $p$ Galois representation}
\label{sec:Frey}

For an elliptic curve $E$ over a number field $K$ and a rational prime $p$
we write
\[
\overline{\rho}_{E,p} \; : \; G_K \rightarrow \Aut(E[p]) \cong \GL_2(\F_p)
\]
for the representation induced by the action of $G_K$ on the
$p$-torsion $E[p]$.
We make repeated use of the following lemma. 
\begin{lemma}\label{lem:inertiaGeneral}
Let $E$ be an elliptic curve over $K$ with $j$-invariant $j$.
Let $p \ge 5$ be a rational prime and write $\overline{\rho}=\overline{\rho}_{E,p}$.
Let $\fq \nmid p$ be a prime of $K$.
\begin{enumerate}
\item[(i)] If $\ord_\fq(j) \ge 0$ (i.e. $E$ has potentially good
reduction at $\fq$) then $\# \overline{\rho}(I_\fq) \mid 24$.
\item[(ii)] Suppose $\ord_\fq(j) <0$ (i.e. $E$ has potentially multiplicative
reduction at $\fq$).
\begin{itemize} 
\item If $p \nmid \ord_\fq(j)$ then
$\# \overline{\rho}(I_\fq)=p$ or $2p$.
\item If $p \mid \ord_\fq(j)$ then
$\# \overline{\rho}(I_\fq)=1$ or $2$.
\end{itemize}
\end{enumerate}
\end{lemma}
\begin{proof}
For (i) see 
\cite[Introduction]{kraus2}. 
For (ii) we
suppose first that $E$
has split multiplicative reduction at $\fq$. 
As $\fq \nmid p$ and $E$ is semistable at $\fq$, inertia at $\fq$
acts unipotently on $E[p]$, and thus $\# \overline{\rho}(I_\fq) \mid p$.
From the theory of the Tate
curve \cite[Proposition V.6.1]{SilvermanII}, we know
that $p \mid \# \overline{\rho}(I_\fq)$ if and only if 
$p \nmid \ord_\fq(j)$. This proves (ii) if $E$ has split
multiplicative reduction. Suppose now that $E$ has potentially
multiplicative reduction. Then $E$ is a quadratic twist of 
an elliptic curve $E^\prime$ with split multiplicative reduction.
Thus $\overline{\rho}=\phi \otimes \overline{\rho^\prime}$,
where $\overline{\rho^\prime}=\overline{\rho}_{E^\prime,p}$ and $\phi$
is a quadratic character. Part (ii) follows.
\end{proof}

\medskip

Let $S$ and $T$ be as in the Introduction.
We suppose once and for all that
$T \ne \emptyset$.
Let $p$
be an odd prime, and let
$(a,b,c) \in K^3$ be a non-trivial solution to the Fermat equation
\eqref{eqn:Fermat}.
We scale the solution $(a,b,c)$ so that it is
integral. As the class number might not
be $1$, we cannot suppose that $a$, $b$, $c$ are coprime.
However, following \cite{FS}, we may suppose that the
ideal generated by $a$, $b$, $c$ belongs to a finite set
as we now explain.
For a non-zero ideal $\ga$ of $\ZK$, we denote by
$[\ga]$ the class of $\ga$ in the class group $\Cl(K)$.
Let
\begin{equation}\label{eqn:cG}
\cG_{a,b,c}:=a \ZK+b\ZK+c\ZK \, .
\end{equation}
and let $[a,b,c]$ denote
the class of $\cG_{a,b,c}$ in $\Cl(K)$.
We exploit the well-known fact
(e.g. \cite[Theorem VIII.4]{CF})
that every ideal class contains
infinitely many prime ideals.
Let $\mathfrak{c}_1,\dotsc,\mathfrak{c}_h$
be the ideal classes of $K$.
For each class $\mathfrak{c}_i$, we choose
(and fix) a prime ideal $\fp_i \nmid 2$ of smallest possible
norm representing $\mathfrak{c}_i$.
The set $\cH$ denotes our fixed choice of odd prime ideals
representing the class group:
$\cH=\{\fp_1,\dots,\fp_h\}$.
By \cite[Lemma 3.2]{FS}, we may scale $(a,b,c)$ so that
it remains integral, but $\cG_{a,b,c} \in \cH$.
We shall henceforth suppose $\cG_{a,b,c}=\fp \in \cH$.
Associated to $(a,b,c)$ is the Frey curve
\begin{equation}\label{eqn:Frey}
E=E_{a,b,c} \; : \; Y^2=X(X-a^p)(X+b^p).
\end{equation}
We write $\overline{\rho}=\overline{\rho}_{E,p}$.

The following  is Lemma 3.7 of \cite{FS}, but as it is crucial
to everything that follows we include a proof here.
\begin{lemma}\label{lem:mPImage}
Let $\mP \in T$ and suppose $p > 4\ord_\mP(2)$. Then
\begin{enumerate}
\item[(i)]
$E$ has potentially multiplicative reduction at $\mP$;
\item[(ii)] $p \mid \# \overline{\rho}(I_\mP)$ where $I_\mP$
denotes the inertia subgroup of $G_K$ at $\mP$.
\end{enumerate}
\end{lemma}
\begin{proof}
Since $\cG_{a,b,c}=\fp \nmid 2$, we know $\mP$ divides at most one of $a$, $b$, $c$. 
By definition of $T$, the residue field of $\mP$
is $\F_2$. If $\mP \nmid abc$ then $0=a^p+b^p+c^p \equiv 1+1+1 \pmod{\mP}$,
giving a contradiction. 
We see that $\mP$ divides precisely one of $a$, $b$, $c$. We
permute $a$, $b$, $c$ so that $\mP \mid b$; such a permutation
corresponds to twisting $E$ by $\pm 1$, and so does not affect $j$.
Now the expression for $j$ in terms of $a$, $b$, $c$ is
\begin{equation}\label{eqn:jinvariant}
j=2^8 \cdot \frac{(b^{2p}-a^p c^p)^3}{a^{2p} b^{2p} c^{2p}}.
\end{equation}
It follows that $\ord_\mP(j)=8 \ord_\mP(2)-2p \ord_{\mP}(b)$.
As $p> 4 \ord_\mP(2)$ we have that $\ord_\mP(j)<0$ and so
$E$ has potentially multiplicative reduction at $\mP$. Moreover,
$p \nmid \ord_\mP(j)$. 
The lemma follows from Lemma~\ref{lem:inertiaGeneral}.
\end{proof}

\begin{lemma}\label{lem:mimage}
Suppose $p \ge 5$ and $\fp \nmid p$. Then
$\# \overline{\rho}(I_\fp) \mid 24$.
\end{lemma}
\begin{proof}
As $\cG_{a,b,c}=\fp$ we know that $\ord_\fp(a)$, $\ord_\fp(b)$, $\ord_\fp(c)$
are all positive. Moreover, as $a^p+b^p+c^p=0$, we have 
that at least two of $\ord_\fp(a)$, $\ord_\fp(b)$, $\ord_\fp(c)$ are equal.
Permuting $a$, $b$, $c$ (which twists $E$ by $\pm 1$ and so does
not affect the image of inertia at $\fp \nmid 2$) we may suppose
\[
\ord_\fp(a)=\ord_\fp(c)=k, \qquad \ord_\fp(b)=k+t
\]
where $k \ge 1$ and $t \ge 0$. If $t=0$ then
from \eqref{eqn:jinvariant} we have $\ord_\fp(j) \ge 0$
and so the lemma follows from Lemma~\ref{lem:inertiaGeneral}.
Thus suppose that $t \ge 1$. Then $\ord_\fp(j)=-2pt$ and so
by Lemma~\ref{lem:inertiaGeneral} we have $\overline{\rho}(I_\fp)=1$
or $2$, completing the proof.
\end{proof}

\begin{lemma}\label{lem:Serre}
The Frey curve $E$ is semistable away from $S \cup \{\fp\}$,
where $\fp=\cG_{a,b,c}$.
Suppose $p \ge 5$ and not divisible by any $\fq \in S \cup \{\fp\}$.
The determinant of $\overline{\rho}$ is the mod $p$ cyclotomic character.
Its Serre conductor $\fN$ is supported on $S \cup \{\fp\}$
and belongs to a finite set
that depends only on the field $K$. The representation
$\overline{\rho}$ is odd (in the sense of Section \ref{sec:Reps}) and is finite flat at every $\fq$ over $p$. 
\end{lemma}
\begin{proof}
The statement about the determinant
 is a well-known consequence of the theory of 
the Weil pairing on $E[p]$. This immediately implies oddness. Let $\fq \notin S \cup \{\fp\}$
be a prime of $K$.
Let $c_4$ and $\Delta$ denote the usual invariants of the model $E$
given in \eqref{eqn:Frey}. These are given by the formulae
\[
c_4=2^4 (b^{2p}-a^p c^p), \qquad
\Delta=2^4 a^{2p} b^{2p} c^{2p}.
\]
It follows from $\fq \notin S \cup \{\fp\}$ (together with
the relation $a^p+b^p+c^p=0$) that $\fq$ cannot divide
both $c_4$ and $\Delta$. Thus the given model is minimal,
and $E$ is semistable at $\fq$. Moreover, $p \mid \ord_\fq(\Delta)$.
It follows (c.f. \cite{SerreDuke}) that $\overline{\rho}$ is unramified 
at $\fq$ if $\fq \nmid p$ and finite flat at $\fq$ if $\fq \mid p$.
It remains to show that the set of possible Serre conductors
$\fN$ is finite. These can only be divisible
by primes $\fq \in S \cup \{ \fp\}$. Moreover $\fN$ 
divides the conductor $N$ of $E$, thus 
$\ord_\fq(\fN) \le \ord_\fq(N) \le 2 + 3 \ord_\fq(3)+6 \ord_\fq(2)$
by \cite[Theorem IV.10.4]{SilvermanII}. It follows that the list of
possible Serre conductors is finite. Moreover as $\fp \in \cH$
and the set $\cH$ depends only on $K$,
this list of Serre conductors depends only on $K$.
\end{proof}

\section{Surjectivity of $\overline{\rho}$}

To apply Conjecture~\ref{conj:Serre} to the mod $p$ representation
$\overline{\rho}$ of the Frey curve $E$, we need to show that
$\overline{\rho}$ is absolutely irreducible. We have been unable to 
find a theorem in the literature that immediately implies this.
However, guided by the work of Momose \cite{Momose}, 
Kraus \cite{KrausIrred} and David \cite{DavidI} (all
relying on Merel's uniform boundedness theorem \cite{Merel}),
we prove the following result which is sufficient for our purpose.
\begin{prop}\label{prop:irred}
Let $L$ be a Galois number field and let $\fq$ be a prime of $L$.
There is a constant $B_{L,\fq}$ such that the following is true.
Let $p> B_{L,\fq}$ be a rational prime. Let $E/L$ be
an elliptic curve 
that is semistable at all $\gp \mid p$ and
having potentially multiplicative reduction at $\fq$.
Then $\overline{\rho}_{E,p}$
is irreducible.
\end{prop}

Before proving Proposition~\ref{prop:irred} we apply it
to the Frey curve. 
\begin{corollary}\label{cor:surj}
Let $K$ be a number field, and suppose 
(in the notation of the Introduction) that $T \ne \emptyset$.
There is a constant $C_K$ such that if $p>C_K$ and $(a,b,c) \in \ZK^3$
is a non-trivial solution to the Fermat equation with exponent $p$,
and scaled so that $\cG_{a,b,c} \in \cH$, then $\overline{\rho}_{E,p}$
is surjective, where $E$ is the Frey curve given in \eqref{eqn:Frey}.
\end{corollary}
\begin{proof}
We know from Lemma~\ref{lem:mPImage} that $E$ has potentially
multiplicative reduction at $\mP \in T$. Moreover, from the 
proof of Lemma~\ref{lem:Serre}, we know that $E$ is semistable
away from the primes above $2$ and those contained in $\cH$.
Let $L$ be the Galois closure of $K$, and let $\fq$ be a prime
of $L$ above $\mP$.
Applying Proposition~\ref{prop:irred} 
we see that there is a constant
$B_{L,\fq}$ such that for $p>B_{L,\fq}$ we have that 
$\overline{\rho}_{E,p} (G_{L})$
is irreducible. Now $\fq \mid \mP \mid 2$ and so  $B_{L,\fq}$
depends only on $K$ and we denote it by $C_K$. 
We enlarge $C_K$ if needed so that for $C_K> 4 \ord_\mP(2)$.
It follows 
from Lemma~\ref{lem:mPImage} that the image of $\overline{\rho}_{E,p}$
contains an element of order $p$. Any subgroup of $\GL_2(\F_p)$
that contains an element of order $p$ is either reducible or contains
$\SL_2(\F_p)$. It follows for $p >C_K$ that the image in fact contains
$\SL_2(\F_p)$. Moreover, again enlarging $C_K$ if necessary, 
we may suppose that $K \cap \Q(\zeta_p)=\Q$ for $p>C_K$. Thus $\chi_p=\det(\overline{\rho}_{E,p})$ is surjective, completing the proof.
\end{proof}

\subsection{Proof of Proposition~\ref{prop:irred}}

Suppose $\overline{\rho}_{E,p}$ is reducible. Thus,
\begin{equation}\label{eqn:matrix}
\overline{\rho}_{E,p} \sim 
\begin{pmatrix}
\lambda & * \\
0 & \lambda^\prime
\end{pmatrix} \, ,
\end{equation}
where $\lambda$, $\lambda^\prime \, : \, G_L \rightarrow \F_p^*$ are characters,
and
 $\lambda \lambda^\prime= \det(\overline{\rho}_{E,p})=\chi_p$ is the mod $p$
cyclotomic character.
We suppose from now on that $p$ is unramified in $L$ and that
$E$ is semistable at all $\gp \mid p$.

\begin{lemma}
Write $\sigma_\fq \in G_L$ for a Frobenius element at $\fq$.
Then $\lambda^2(\sigma_\fq)$, ${\lambda^\prime}^2(\sigma_\fq)$
are (up to reordering) congruent to $1$, $\norm(\fq)^2$ modulo $p$.
\end{lemma}
\begin{proof}
Write $D_\fq$ for the decomposition subgroup at $\fq$.
 As $E$ has potentially multiplicative reduction at $\fq$,
we know that $\overline{\rho}_{E,p} \vert_{D_\fq}$ is 
up to semisimplification equal to $\phi \oplus \phi \cdot \chi_p$,
where $\phi$ is at worst a quadratic character. The
lemma follows since $\chi_p(\sigma_\fq) \equiv \norm(\fq) \pmod{p}$.
\end{proof}

From \eqref{eqn:matrix}, there is a non-zero $P \in E[p]$
such that $\sigma(P)=\lambda(\sigma) P$ for $\sigma \in G_K$. 
Replacing $E$ by $p$-isogenous $E/\langle P \rangle$ results
in swapping the two characters $\lambda$, $\lambda^\prime$
in \eqref{eqn:matrix}. This allows us to suppose from
now on that $\lambda^2(\sigma_\fq) \equiv 1 \pmod{p}$.

\begin{lemma}\label{lem:David1}
The character $\lambda^{12}$ is unramified away from 
the primes above $p$. 
Let $\gp \mid p$ be a prime of $K$.
Then
\[
\lambda^{12} \vert_{I_\gp}= \left(\chi_p \vert_{I_\gp} \right)^{s_\gp}
\]
where $s_\gp \in \{0,12\}$.
\end{lemma}
\begin{proof}
The first part of the lemma is Proposition 1.4 and 1.5 of \cite{DavidI}.
The second part is derived in \cite[Proposition 2.1]{FSIrred} from results
found in \cite{DavidI}.
\end{proof}

Now let $G=\Gal(L/\Q)$. As $L/\Q$ is Galois, $G$ acts transitively
on the primes $\gp \mid p$. Write $\gp_0 \mid p$. 
For $\tau \in G$ we write $s_\tau \in \{0,12\}$ for the 
integer $s_\gp$ associated to $\gp=\tau^{-1}(\gp_0)$ in Lemma~\ref{lem:David1}.

\begin{lemma}{(David \cite[Proposition 2.6]{DavidI})}\label{prop:lam12}
Let $\alpha \in L$ be non-zero. Suppose $\ord_\gp(\alpha)=0$
for all $\gp \mid p$. Then
\[
\prod_{\tau \in G} \tau(\alpha)^{s_\tau}  
\equiv
\prod \left(\lambda^{12}(\sigma_\fr) \right)^{\upsilon_\fr(\alpha)}
\pmod{\gp_0},
\]
where the product on the right-hand side 
is taken over all prime $\fr$ in the support of $\alpha$.
\end{lemma}

We now choose a positive integer $r$ so that $\fq^r$ is principal
and write $\alpha \Z_L=\fq^r$. Since $\lambda^{12}(\sigma_\fq)=1$
we see that
\[
\prod_{\tau \in G} \tau(\alpha)^{s_\tau} \equiv 1 \pmod{\gp_0}.
\]
The left-hand side belongs to a finite set $\cA$ that depends
only on $L$ and $\fq$, since $s_\tau=0$ or $12$
for any $\tau \in G$. Moreover, if we denote the left-hand side
by $\beta$ then $p$ divides $\norm(\gp_0)$ which in turn divides
$\norm(\beta-1)$. We choose $B_{L,\fq} > \norm(\beta-1)$ for
all $\beta \in \cA$ with $\beta \ne 1$. Then for $p> B_{L,\fq}$
we deduce that $\beta=1$. Thus $s_\tau=0$ for all $\tau \in G$.
It follows from Lemma~\ref{lem:David1}  that the character
$\lambda^{12}$ is unramified at all places of $L$. Hence
there is an extension $M/L$ of degree $12 \cdot h_L$ where $h_L$
is the class number of $L$ such that $\lambda \vert_{G_M}=1$.
It follows from \eqref{eqn:matrix} that $E$ has a point
of order $p$ over $M$. Finally applying Merel's uniform boundedness
theorem \cite{Merel}, shows that $p$ is bounded by a constant that
depends only on the degree $[M:\Q]=12 \cdot h_L \cdot [L:\Q]$
completing the proof.

\section{Applying the Conjectures}
This section is devoted to the proof of the following proposition.
\begin{proposition}\label{prop:elliptic}
Let $K$ be a number field. Assume Conjectures~\ref{conj:Serre}
and~\ref{conj:ES}. Suppose, in the notation of Section~\ref{sec:Frey},
that $T \ne \emptyset$. Then there is a constant $B_K$
depending only on $K$ such that the following holds.
Let $(a,b,c) \in \ZK^3$
being a non-trivial solution to the Fermat equation with
exponent $p>B_K$, and we suppose that it is scaled so that
$\cG_{a,b,c}=\fp \in \cH$.  Let $E/K$ be the associated Frey curve defined in (\ref{eqn:Frey}). Then there is an elliptic curve
$E^\prime/K$ such that the following hold:
\begin{enumerate}
\item[(i)] $E^\prime$ has good reduction away from $S \cup \{\fp\}$,
and potentially good reduction away from $S$.
\item[(ii)] $E^\prime$ has full $2$-torsion.
\item[(iii)] $\overline{\rho}_{E,p} \sim \overline{\rho}_{E^\prime,p}$.
\item[(iv)] for $\mP \in T$ we have $\ord_\mP(j^\prime)<0$ where $j^\prime$
is the $j$-invariant of $E^\prime$.
\end{enumerate}
\end{proposition}
Assume the hypotheses of the proposition.
We suppose that $p$ is suitably large, and so
by Corollary~\ref{cor:surj}, the representation $\overline{\rho}_{E,p}$
is surjective. We now apply Conjecture~\ref{conj:Serre} and deduce the existence
of a weight $2$ mod $p$ eigenform $\theta$ over $K$ of level $\fN$, with $\fN$ as in Lemma~\ref{lem:Serre}, such that 
for all primes $\q$ coprime to $p\fN$, we have 
\[
\Tr(\overline{\rho}_{E,p}
({\Frob}_\q)) = \theta(T_\q). 
\]
Since there are only finitely many possible levels $\fN$, see Lemma~\ref{lem:Serre}, we can take $p$ large enough to guarantee that, see Proposition \ref{prop: lifting}, 
for any level $\fN$, there will be a weight $2$ complex eigenform $\ff$ with level $\fN$ that is a lift 
of $\theta$. Observe that the list of such eigenforms $\ff$ is finite and
depends only on $K$ (and not on $p$ or the solution $(a,b,c)$).
Thus every constant that depends later of these eigenforms
depends only on $K$.

Next we show that if $p$ is sufficiently large then $\Q_\ff=\Q$.
The idea here is due to Mazur, though apparently unpublished.
It can be found in \cite[Section 9]{Siksek}.

\begin{lemma}
Suppose $\Q_\ff \ne \Q$.
There is a constant $C_\ff$ depending only on $\ff$ such that $p < C_\ff$.
\end{lemma}
\begin{proof}
Choose and fix a prime $\fq$ of $K$ such that $\fq \notin S \cup \{ \fp\}$ and 
$\ff(T_\q) \notin \Q$. If $\fq \mid p$ then $p \mid \norm(\fq)$
and so $p$ is bounded. Thus we may suppose that $\fq \nmid p$.
Now $E$ has either good or multiplicative reduction at $\fq$ (since $E$
is semistable away from the primes in $S \cup \{ \fp\}$). 
Note that
\[
\Tr(\overline{\rho}_{E,p}(\Frob_\fq))=
\begin{cases}
\pm (\norm(\fq)+1) & \text{if $E$ has multiplicative reduction at $\fq$}\\
a_\fq(E) & \text{if $E$ has good reduction at $\fq$}.
\end{cases}
\]
In particular, this trace belongs to a finite list of rational integers that 
depends only
on $\fq$. However, there is prime ideal $\mathfrak{p}$ of $\Q_\ff$ over $p$ such that
\[
\Tr(\overline{\rho}_{E,p}(\Frob_\fq)) \equiv \ff(T_\fq)  \pmod{\mathfrak{p}}.
\]
As $\ff(T_\q) \notin \Q$, the difference between the
two sides is non-zero and belongs to a finite set. 
As $\mathfrak{p} \mid p$, the norm
of the difference is divisible by $p$. This gives an upper bound
on $p$ that depends only on $\ff$. 
\end{proof} 
Note that if $\Q_\ff=\Q$ then the above argument fails as the difference
might be zero.

By supposing that $p$ is sufficiently large, we may henceforth
suppose that $\Q_\ff=\Q$. The fact that $\overline{\rho}_{E,p}$ is irreducible 
implies that $\ff$ is non-trivial.  If $\ff$ is not new, we replace it with an equivalent 
new eigenform that is of smaller level. Thus we can assume that 
$\ff$ is new and has level $\fN'$ dividing $\fN$. 
By Conjecture~\ref{conj:ES},
$\ff$ either has an associated elliptic curve $E_\ff/K$ of conductor
$\fN'$, or has an associated fake elliptic curve $A_\ff/K$ of conductor $\fN'^2$.
\begin{lemma}
If $p>24$ then $\ff$ has an associated elliptic curve $E_\ff$.
\end{lemma}
\begin{proof}
This is another point where we make use of our assumption $T \ne \emptyset$.
Let $\mP \in T$. We know from Lemma~\ref{lem:mPImage} that
$p \mid \#\overline{\rho}_{E,p}(I_\mP)$. If $\ff$ corresponds
to a fake elliptic curve $A_\ff$, then it follows from Theorem~\ref{thm:goodredn}
that $\#\overline{\rho}_{A_\ff,p}(I_\mP) \le 24$ where $\rho_{A_\ff,p}$ is the $2$-dimensional 
representation defined in (\ref{eqn:fec-rep}). As 
$\overline{\rho}_{E,p} \sim \overline{\rho}_{A_\ff,p}$ we have a contradiction.
\end{proof}

We may henceforth suppose that $\overline{\rho}_{E,p} \sim \overline{\rho}_{E^\prime,p}$ where $E^\prime=E_\ff$ is an elliptic curve of conductor $\fN'$ dividing $\fN$.

\begin{lemma}\label{lem:FTT}
If $E^\prime$ does not have full $2$-torsion, and is not $2$-isogenous
to an elliptic curve with full $2$-torsion, then $p \le C_{E^\prime}$.
\end{lemma}
\begin{proof}
By Lemma~\ref{lem:2tors} (below)
there infinitely many primes 
$\fq$
such that $\# E^\prime(\F_\fq) \not\equiv 0 \pmod{4}$. Fix
such a prime $\fq \notin S \cup \{\fp\}$. Now
if $\fq$ is a prime of good reduction for $E$, then
$\#E(\F_\fq) \equiv E^\prime(\F_\fq) \pmod{p}$. Note that $\#E(\F_\fq)$
is divisible by $4$ as the Frey curve $E$ has full $2$-torsion.
Thus the difference $\# E(\F_\fq)-\#E^\prime(\F_\fq)$, which is
divisible by $p$, is non-zero. Moreover, this difference belongs to
a finite set depending on $\fq$, and so $p$ is bounded. We may
therefore suppose that $E$ has multiplicative reduction at $\fq$.
In this case, comparing traces of Frobenius at $\fq$ we have
\[
\pm (\norm(\fq)+1) \equiv a_\fq(E^\prime) \pmod{p}.
\]
Again the difference is non-zero and depends only on $\fq$,
giving a bound for $p$.
\end{proof}
If $E^\prime$ is $2$-isogenous to an elliptic curve $E^{\prime\prime}$
then (as $p \ne 2$) then the isogeny induces
an isomorphism $E^\prime[p] \cong E^{\prime\prime}[p]$ of Galois modules.
Thus $\overline{\rho}_{E,p} \sim \overline{\rho}_{E^\prime,p}
\sim \overline{\rho}_{E^{\prime\prime},p}$. Hence we may,
after possibly replacing $E^\prime$ by $E^{\prime\prime}$, suppose
that $E^\prime$ has full $2$-torsion. To complete the proof the 
proposition, we need to show that $E^{\prime}$ has potentially good
reduction at $\fp$, and that $\ord_\mP(j^\prime)<0$ for 
$\mP \in T$. Recall by Lemmas~\ref{lem:mPImage} and ~\ref{lem:mimage}
that $p \mid \# \overline{\rho}_{E,p}(I_\mP)$ for $\mP \in T$,
and that $\# \overline{\rho}_{E,p}(I_\fp) \le 24$. 
As $\# \overline{\rho}_{E,p}(I_\mP)= \# \overline{\rho}_{E^\prime,p}(I_\mP)$
we deduce from Lemma~\ref{lem:inertiaGeneral} that $\ord_\mP(j^\prime)<0$
for $\mP \in T$.
Finally if $E^\prime$ has potentially multiplicative reduction at $\fp$
then for every $p> \lvert \ord_\fp(j^\prime) \rvert$ we have,
by Lemma~\ref{lem:inertiaGeneral}, that 
$p \mid \# \overline{\rho}_{E^\prime,p}(I_\fp)$, giving a contradiction
for large $p$. 

\subsection{$2$-Torsion of Elliptic Curves}\label{sec:2tors}
To complete the proof Lemma~\ref{lem:FTT} we need
the following result which is stated as a fact in
\cite[Section 3]{KrausFermat}. We are grateful to 
Nicolas Billerey for pointing out that this is 
a special case of a theorem of Katz
\cite{Katz}. 
\begin{lemma}\label{lem:2tors}
Let $E$ be an elliptic curve over a number field $K$.
Suppose that $4 \mid \#E(\F_\fq)$ for all primes $\fq$
of sufficiently large norm.
Then either $E$ has full $2$-torsion, or it
is  $2$-isogenous to some elliptic curve $E^\prime$
having full $2$-torsion.
\end{lemma}
\begin{proof}
By \cite[Theorem 2]{Katz} there is an elliptic curve $E^\prime/K$
isogenous to $E$ such that $4 \mid \#E^\prime(K)_{\mathrm{tors}}$.
If $E^\prime$ has full $2$-torsion then we are finished.
Otherwise $E^\prime$ has some $K$-point $P$ of order $4$. 
The points of order $2$ on $E^\prime$ are $2P$, $Q$, $R$ (say)
where $Q$ and $R$ are Galois conjugates, related by $R=Q+2P$. The points
of order $2$ on  the $2$-isogenous curve 
$E^\prime/\langle 2P \rangle$ are $P+\langle 2P\rangle$,
$Q+\langle 2P \rangle=R+\langle 2P \rangle$ and $P+Q+\langle 2P \rangle$.
These are clearly individually fixed by the action of $G_K$. 
\end{proof}

\section{Proof of Theorem~\ref{thm:FermatGen}}

 We apply
Proposition~\ref{prop:elliptic} which yields an elliptic curve $E^\prime/K$
with full $2$-torsion and potentially good reduction outside $S$
whose $j$-invariant $j^\prime$ satisfies 
$\ord_\mP(j^\prime) <0$ for all $\mP \in T$.
Write
\[
E^\prime \; : \; Y^2=X(X-e_1)(X-e_2)
\]
with $e_1$, $e_2 \in \ZK$. Let $\lambda=e_1/e_2$. Let $\lambda^\prime$
be any of the following six expressions (which are known
as the $\lambda$-invariants of $E^\prime$):
\[
\lambda, \quad 1/\lambda, \quad 1-\lambda, 
\quad 1/(1-\lambda),\quad 
\lambda/(\lambda-1),\quad
(\lambda-1)/\lambda.
\]
Then
\begin{equation}\label{eqn:jp}
j^\prime=2^8 \cdot \frac{({\lambda^\prime}^2-\lambda^\prime+1)^3}{{\lambda^\prime}^2 (1-\lambda^\prime)^2} \, .
\end{equation}
Let $\fq \notin S$ be a prime of $K$. As $E^\prime$ has potentially good
reduction at $\fq$, we know that $\ord_\fq(j^\prime) \ge 0$.
Thus $\lambda^\prime$ is the root of a degree six monic polynomial
with coefficients that are $\fq$-integral.
It immediately follows that $\ord_\fq(\lambda^\prime) \ge 0$. This is true
for both $\lambda^\prime=\lambda$ and $\lambda^\prime=1/\lambda$, 
thus $\lambda \in \OO_S^\times$. Moreover, letting $\mu=1-\lambda$ we see
that $\mu \in \OO_S^\times$, hence $(\lambda,\mu)$ is a solution to the
the $S$-unit equation \eqref{eqn:sunit}.
Suppose, as in the statement of the theorem, that for every
such solution $(\lambda,\mu)$ there some $\mP \in T$
such that 
$t:=\max\{\lvert \ord_\mP(\lambda)\rvert, \lvert \ord_\mP(\mu)\rvert\}
\le \ord_\mP(2)$. If $t=0$ then it follows from \eqref{eqn:jp}
with $\lambda^\prime=\lambda$ that $\ord_\mP(j^\prime)>0$
giving a contradiction. Thus $t>0$. 
Now the relation $\lambda+\mu=1$
forces either $\ord_\mP(\lambda)=\ord_\mP(\mu)=-t$, 
or $\ord_\mP(\lambda)=0$ and $\ord_\mP(\mu)=t$, or $\ord_\mP(\lambda)=t$ and $\ord_\mP(\mu)=0$.
Thus $\ord_\mP(\lambda \mu)=-2t<0$ or $\ord_\mP(\lambda \mu)=t>0$.
But
\[
j^\prime
=2^8 \cdot (1-\lambda \mu)^3 \cdot (\lambda \mu)^{-2} \, ,
\]
which shows, either way, that $\ord_\mP(j^\prime) = 8 \ord_\mP(2)-2t \ge 0$
giving a contradiction. This completes the proof.


\end{document}